\documentclass[11pt]{amsart}

  \usepackage[utf8]{inputenc}

  \usepackage{amsmath}
  \usepackage{amsfonts}
  \usepackage{amssymb}
  \usepackage{amsthm}

  \usepackage{mathrsfs}
  \usepackage{tikz}

  \usepackage[a4paper, bindingoffset=1cm]{geometry}

  \usetikzlibrary{matrix,arrows,calc}
  
\usepackage{xfrac}
\usepackage{amsmath}
\usepackage{amstext}
\usepackage{amssymb}
\usepackage{amsthm}
\usepackage{mathtools}
\usepackage{mathrsfs}
\usepackage{microtype}
\usepackage{bbm}
\newcommand*{\Scale}[2][4]{\scalebox{#1}{$#2$}}%
\newtheorem{theo}{Theorem}

\newtheorem{prop}[theo]{Proposition}
\newtheorem{cor}[theo]{Corollary}

\newtheorem*{theo*}{Theorem}

\newtheorem{rem}[theo]{Remark} 

\newcommand{\matN}{\ensuremath {\mathbb{N}}}
\newcommand{\matQ}{\ensuremath {\mathbb{Q}}}
\newcommand{\matR} {\ensuremath {\mathbb{R}}}
\newcommand{\matC} {\ensuremath {\mathbb{C}}}

\newcommand{\matZ} {\ensuremath {\mathbb{Z}}}

\renewcommand{\epsilon}{\ensuremath \varepsilon}
\renewcommand{\bar}[1]{\ensuremath \overline{#1}}

\begin{document}
 \title[Degree zero part of the motivic polylogarithm on abelian schemes]{The realization of the degree zero part of the motivic polylogarithm on abelian schemes in Deligne-Beilinson cohomology}

\author{Danny Scarponi}




\maketitle

  
  
  
  
\begin{abstract}
 We use Burgos' theory of arithmetic Chow groups to  exhibit a realization of the degree zero part of the polylogarithm on abelian schemes in Deligne-Beilinson cohomology.
\end{abstract}

\section{\textbf{INTRODUCTION}}
\subsection{The degree zero part of the motivic polylogarithm}
In \cite{kingsross}, G. Kings and D. R\"ossler have given a simple axiomatic description of the degree zero part of the polylogarithm on abelian schemes. We briefly recall it here.

In \cite{soul85}, C. Soul\'e has defined motivic cohomology for any variety $V$ over a field
\[
 H^{i}_{\mathcal{M}}(V,j):=\textnormal{Gr}^j_{\gamma}K_{2j-i}(V)\otimes \matQ.
\]
Now let $\pi :\mathcal{A}\rightarrow S$ be an abelian scheme of relative dimension $g$, let $\epsilon: S \rightarrow\mathcal{A}$ be the zero section, $N>1$ an integer and let 
$\mathcal{A}[N]$ be the finite group scheme of $N$-torsion points. Here $S$ is smooth over a subfield $k$ of the complex numbers.
For any integer $a>1$ and any $W\subseteq \mathcal{A}$ open sub-scheme such that 
\[
 j:[a]^{-1}(W)\hookrightarrow W
\]
is an open immersion (here $[a]:\mathcal{A}\rightarrow \mathcal{A}$ is the $a$-multiplication on $\mathcal{A}$), the trace map with respect to $a$ is defined as
\begin{equation}
\begin{tikzpicture}[description/.style={fill=white,inner sep=2pt}, ciao/.style={fill=red,inner sep=2pt}]
\matrix (m) [matrix of math nodes, row sep=3.5em,
column sep=2.5em, text height=1.5ex, text depth=0.25ex]
{ \textnormal{tr}_{[a]}:H^{\cdot}_{\mathcal{M}}(W,*) & H^{\cdot}_{\mathcal{M}}([a]^{-1}(W),*)
& H^{\cdot}_{\mathcal{M}}(W,*) \\
  };
	\path[->,font=\scriptsize] 
		(m-1-1) edge node[above] {$ j^* $} (m-1-2)
		(m-1-2) edge node[above] {$ [a]_* $} (m-1-3)
		;
\end{tikzpicture}
\end{equation}
For any integer $r$ we let
\[
 H^{\cdot}_{\mathcal{M}}(W,*)^{(r)}:=\{ \psi\in H^{\cdot}_{\mathcal{M}}(W,*)|( \textnormal{tr}_{[a]}-a^r\textnormal{Id})^k\psi=0 \textnormal{ for some } k\geq 1 \}
\]
be the generalized eigenspace of $ \textnormal{tr}_{[a]}$ of weight $r$.

 Then the zero step of the motivic polylogarithm
 is a class in motivic cohomology
\[
 \textnormal{pol}^0\in H^{2g-1}_{\mathcal{M}}(\mathcal{A}\setminus\mathcal{A}[N],g).
\]
To describe it more precisely, consider the residue map along $\mathcal{A}[N]$
\[
 H^{2g-1}_{\mathcal{M}}(\mathcal{A}\setminus\mathcal{A}[N],g)\rightarrow H^0_{\mathcal{M}}(\mathcal{A}[N]\setminus \epsilon(S),0).
\]
This map induces an isomorphism
\[
 H^{2g-1}_{\mathcal{M}}(\mathcal{A}\setminus\mathcal{A}[N],g)^{(0)}\cong H^0_{\mathcal{M}}(\mathcal{A}[N]\setminus \epsilon(S),0)^{(0)}
\]
(see Corollary 2.2.2 in \cite{kingsross}) and $\rm{pol}^0$ is the unique element mapping to the fundamental class of $\mathcal{A}[N]\setminus \epsilon(S)$.

Now let us consider the map $\textnormal{cyc}_{\textnormal{an}}$ defined as the composition
\begin{equation} 
\begin{tikzpicture}[scale=0.5][description/.style={fill=white,inner sep=0.5pt}, ciao/.style={fill=red,inner sep=2pt}]
\matrix (m) [matrix of math nodes, row sep=3.5em,
column sep=2em, text height=1.5ex, text depth=0.25ex]
{ H^{2g-1}_{\mathcal{M}}(\mathcal{A}\setminus\mathcal{A}[N],g) &  \textnormal{H}^{2g-1}_{D}((\mathcal{A}\setminus\mathcal{A}[N])_{\matR},\matR(g)) &
\textnormal{H}^{2g-1}_{D,\textnormal{an}}((\mathcal{A}\setminus\mathcal{A}[N])_{\matR},\matR(g)) \\
  };
	\path[->,font=\scriptsize] 
		(m-1-1) edge node[above] {$ \textnormal{cyc} $} (m-1-2)
		(m-1-2) edge node[above] {$ \textnormal{forgetful} $} (m-1-3)
		;
\end{tikzpicture} 
\end{equation} 
where $ \textnormal{cyc}$ is the regulator map into Deligne-Beilinson
cohomology and the second map is the forgetful map from Deligne-Beilinson cohomology
to analytic real Deligne cohomology (see the end of section 2 for the notations used here). 

In \cite{onacan}, V. Maillot and D. R\"ossler constructed a canonical class of currents $\mathfrak{g}_{\mathcal{A}^\vee} $  on $\mathcal{A}$ (cf. Theorem \ref{uniqclass} in Section 5)
which gives rise to  a class in analytic Deligne cohomology
\[
 ([N]^\ast \mathfrak{g}_{\mathcal{A}^\vee}
 -N^{2g} \mathfrak{g}_{\mathcal{A}^\vee})|_{\mathcal{A}\setminus\mathcal{A}[N]}\in \textnormal{H}^{2g-1}_{D,\textnormal{an}}((\mathcal{A}\setminus\mathcal{A}[N])_{\matR},\matR(g)).
\]
This element is represented by  $\frac{\gamma}{(2\pi i)^{1-g}}$, where $\gamma$ is any smooth form on $\mathcal{A}\setminus \mathcal{A}[N]$ 
in the class 
$[N]^*\mathfrak{g}_{\mathcal{A}^\vee}-N^{2g}\mathfrak{g}_{\mathcal{A}^\vee}$.

The main result in \cite{kingsross} is the following.
\begin{theo} \label{KR}
 We have
 \[
  -2   \cdot \textnormal{cyc}_{\textnormal{an}} (\textnormal{pol}^0)=([N]^\ast \mathfrak{g}_{\mathcal{A}^\vee}
 -N^{2g} \mathfrak{g}_{\mathcal{A}^\vee})|_{\mathcal{A}\setminus\mathcal{A}[N]}.
 \]
Furthermore the map \[H^{2g-1}_{\mathcal{M}}(\mathcal{A}\setminus\mathcal{A}[N],g)^{(0)}\rightarrow 
\textnormal{H}^{2g-1}_{D,\textnormal{an}}((\mathcal{A}\setminus\mathcal{A}[N])_{\matR},\matR(g))
\] 
induced by $\textnormal{cyc}_{\textnormal{an}} $ is injective.
\end{theo}
\subsection{Our main result}
In this paper we give a refinement of Theorem \ref{KR} supposing $S$ is proper over $k$ (see Corollary \ref{poly} and Theorem \ref{realDB}).

Before stating our result, we recall that in \cite{burgosarith},  Burgos introduced a  complex that naturally computes the Deligne-Beilinson cohomology: this is the 
complex $E^*_{\textnormal{log}}(\cdot)$ of smooth forms with logarithmic singularities along infinity 
(see section 3.1 for its definition). 
\begin{theo}\label{main}
 Let $S$ be proper over $k$.  The class of currents  $[N]^*\mathfrak{g}_{\mathcal{A}^\vee}-N^{2g}\mathfrak{g}_{\mathcal{A}^\vee}$
 has a representative which is smooth on $\mathcal{A}\setminus \mathcal{A}[N]$ 
and has  logarithmic singularities along infinity. Any such $\eta$
   defines an element
 \[\frac{\widetilde{\eta}}{2(2\pi i)^{1-g}}\in \textnormal{Im}\left(\textnormal{cyc}: 
 H^{2g-1}_{\mathcal{M}}(\mathcal{A}\setminus \mathcal{A}[N] ,g) \rightarrow \textnormal{H}^{2g-1}_{D}
 ((\mathcal{A}\setminus \mathcal{A}[N])_\matR, \matR(g))\right)
 \]
which does not depend on the choice of $\eta$. This element verifies
 \[
  -2 \cdot  \textnormal{cyc} (\textnormal{pol}^0)=\frac{\widetilde{\eta}}{(2\pi i)^{1-g}}
 \]
 and
 \[
  \textnormal{forgetful}\left(\frac{\widetilde{\eta}}{(2\pi i)^{1-g}}\right)= ([N]^\ast \mathfrak{g}_{\mathcal{A}^\vee}
 -N^{2g} \mathfrak{g}_{\mathcal{A}^\vee})|_{\mathcal{A}\setminus\mathcal{A}[N]}.
 \]
\end{theo}

\subsection{An outline of the paper}
Let us now give an outline of the contents of each section.

In section 2 we review some notations and definitions coming from Arakelov theory.

In section 3 we recall Burgos' theory of arithmetic Chow groups.

Sections 4,5 and 6 contain the proof of Theorem \ref{main}. This proof combines two arguments. On one hand  we use Burgos' theory 
in order to prove  an interesting intermediate result which relates the classical arithmetic Chow groups to Deligne-Beilinson cohomology (see Theorem \ref{genlem} in section 4).
On the other hand we do some calculations using the current $\mathfrak{g}_{\mathcal{A}^\vee} $
in order to prove that the class of \[T:=[N]^*(\epsilon(S),  \mathfrak{g}_{\mathcal{A}^\vee})-N^{2g}(\epsilon(S),  \mathfrak{g}_{\mathcal{A}^\vee}) \] in 
$\widehat{\textnormal{CH}}^g (\mathcal{A})_\matQ$ is zero (cf. Proposition \ref{zero} in section 5). This proposition will allow us to apply Theorem \ref{genlem} to our case 
(see section 6). 
\section{\textbf{NOTATIONS}}
We begin with a review of some notations and definitions coming from Arakelov theory (see Sections 1,2,3 in \cite{gilsou} for a compendium).
  
  Let $(R,\Sigma,F_\infty)$ be an arithmetic ring i.e. 
\begin{itemize}
 \item $R$ is an excellent regular Noetherian integral domain,
 \item $\Sigma$ is a finite nonempty set of monomorphisms $\sigma:R\rightarrow \matC$,
 \item $F_\infty$ is an anti-linear involution of the $\matC-$algebra $\matC^{\Sigma}:=\matC \underbrace{\times ... \times}_{|\Sigma|}\matC$, such that the diagram
\begin{center}
\begin{tikzpicture}[description/.style={fill=white,inner sep=2pt}, ciao/.style={fill=red,inner sep=2pt}]
\matrix (m) [matrix of math nodes, row sep=3.5em,
column sep=3.5em, text height=1.5ex, text depth=0.25ex]
{ R & \matC^{\Sigma}  \\
  R & \matC^{\Sigma}  \\
  };
	\path[->,font=\scriptsize]
		(m-1-1) edge node[left] {$ \textnormal{Id} $} (m-2-1)
		
		(m-1-2) edge node[auto] {$ F_{\infty} $} (m-2-2)
		
		;
	\path[->,font=\scriptsize] 
		(m-1-1) edge node[above] {$ \delta $} (m-1-2)
		(m-2-1) edge node[above] {$ \delta $} (m-2-2)
		
		;
\end{tikzpicture}
\end{center}
commutes (here by $\delta$ we mean the natural map to the product induced by the family of maps $\Sigma$). 
\end{itemize}
Let $X$ be an arithmetic variety over $R$, i.e. a scheme of finite type over $R$, which is flat, quasi-projective and regular.
As usual we write
\[
 X(\matC):=\coprod_{\sigma\in \Sigma}(X\times_{R,\sigma}\matC)(\matC).
\]
We notice that $F_\infty$ induces an involution $F_\infty:X(\matC)\rightarrow X(\matC)$.
Let $p,q\in \matN$ and $y$ be any cycle in $Z^p(X)$ with $Y:=\textnormal{supp } y$. We denote by:
\begin{itemize}
 \item $D^{p,p}(X_\matR)$ the $\matR$-vector space of real currents 
$\zeta$ on $X(\matC)$ of type $(p,p)$ such that $F^*_\infty \zeta=(-1)^{p}\zeta$,
 \item  $\tilde D^{p,p}(X_\matR)$ the quotient $D^{p,p}(X_\matR)/(\textnormal{Im}\partial 
+\textnormal{Im}\bar\partial)$,
 \item $E^{p,p}(X_\matR)$ the $\matR$-vector space of 
smooth real forms $\omega$ on $X(\matC)$ of type $(p,p)$
such that $F^*_\infty \zeta=(-1)^{p}\omega$,
 \item  $\tilde E^{p,p}(X_\matR)$ the quotient $E^{p,p}(X_\matR)/(\textnormal{Im}\partial 
+\textnormal{Im}\bar\partial)$,
 \item $\textnormal{CH}^q (X)$ the q-th ordinary Chow group of $X$,
 \item $\widehat{\textnormal{CH}}^q (X)$ the q-th arithmetic Chow group of $X$.
\end{itemize}
We also fix the following notations for the analytic real Deligne cohomology and the Deligne-Beilinson cohomology (see \cite{esnault} for the definitions):
\begin{itemize}
\item $\textnormal{H}^{q}_{D, \textnormal{an}}(X(\matC), \matR(p))$ the q-th analytic real Deligne cohomology $\matR$-vector space,
\item $\textnormal{H}^{q}_{D,\textnormal{an}}(X_{\matR}, \matR(p))$ the set $\{\gamma \in \textnormal{H}^{q}_{D,\textnormal{an}}(X(\matC), \matR(p))|F^*_\infty \gamma=(-1)^{p}\gamma \}$,
\item $\textnormal{H}^{q}_{D}(X(\matC), \matR(p))$ the q-th Deligne-Beilinson cohomology $\matR$-vector space,
 \item $\textnormal{H}^{q}_{D,Y}(X(\matC), \matR(p))$ the q-th Deligne-Beilinson cohomology $\matR$-vector space with support in $Y$,
 \item $\textnormal{H}^{q}_{D}(X_{\matR}, \matR(p))$ the set $\{\gamma \in \textnormal{H}^{q}_{D}(X(\matC), \matR(p))|F^*_\infty \gamma=(-1)^{p}\gamma \}$,
 \item $\textnormal{H}^{q}_{D,Y}(X_{\matR}, \matR(p))$ the set $\{\gamma \in \textnormal{H}^{q}_{D,Y}(X(\matC), \matR(p))|F^*_\infty \gamma=(-1)^{p}\gamma \}$.
 \end{itemize}
\section{\textbf{BURGOS' ARITHMETIC CHOW GROUPS}}
For the convenience of the reader we recall in this section some definitions and basic facts of Burgos' arithmetic Chow groups (cf. \cite{burgosarith} for more details).

\subsection{Smooth forms with logarithmic singularities along infinity} Let us start with the definition of smooth differential
forms with logarithmic singularities along infinity.

Let $V$ be a smooth algebraic variety  over $\matC$ and let $D$ be a divisor with normal crossings on $V$. Let us write $W=V\setminus D$ and
let $j:W\hookrightarrow V$ be the inclusion. Let $\mathcal{E}^*_V$ be the sheaf of complex $\textnormal{C}^\infty $ differential forms on $V$ and let
$E^*(V) $ denote $\Gamma(V,\mathcal{E}^*_V)$. The complex of sheaves 
$\mathcal{E}^*_V(\textnormal{log}D)$ is the sub-$\mathcal{E}^*_V$ algebra of $j_*\mathcal{E}^*_W $ generated locally by the sections
\[
 \textnormal{log}z_i\bar{z}_i, \ \frac{dz_i}{z_i}, \ \frac{d\bar{z}_i}{\bar{z}_i}, \ \ \ \textnormal{for} \ i=1, \dots, M,
\]
where $z_1 \cdots z_M=0$ is a local equation of $D$.

Let us write $E^*_V(\textnormal{log}D)=\Gamma(V,\mathcal{E}^*_V(\textnormal{log}D))$ 
and let $E^*_{V,\matR}(\textnormal{log}D)$ be the subcomplex of real forms.

Let $I$ be the category of all smooth compactifications of $V$. That is, an element $(\tilde V_{\alpha},i_\alpha)$ of $I$ is a smooth complex variety 
$\tilde V_{\alpha}$ and an immersion $i_\alpha:V\hookrightarrow \tilde V_{\alpha}$ such that $D_\alpha = \tilde V_{\alpha}-i_\alpha (V)$ is a normal crossing divisor. The morphisms
of $I$ are the maps $f:\tilde V_{\alpha}\rightarrow \tilde V_{\beta}$ such that $f\circ i_\alpha=i_\beta$. The opposed category $I^o$ is directed (see \cite{del}).
The complex of 
smooth differential
forms with logarithmic singularities along infinity  $E^*_{\textnormal{log}}(V)$ is defined as
\[
 E_{\textnormal{log}}^*(V)=\varinjlim_{\alpha\in I^o} E_{\tilde V_{\alpha}}^*(\textnormal{log} D_\alpha)
\]
This complex is a subcomplex of $E^*(V)$ and 
we shall denote by $E^*_{\textnormal{log},\matR}(V)$ the corresponding real  subcomplex.

The complex $E^*_{\textnormal{log}}(V)$ has a natural bigrading 
\[
 E^*_{\textnormal{log}}(V)=\bigoplus E^{p,q}_{\textnormal{log}}(V).
\]
The Hodge filtration of this complex is defined by 
\[
 F^p E^n_{\textnormal{log}}(V)= \bigoplus_{\substack{p'\geq p \\ p'+q'=n}} E^{p',q'}_{\textnormal{log}}(V).
\]
We write $E^*_{\textnormal{log},\matR}(V,p)=(2\pi i)^p E^*_{\textnormal{log},\matR}(V)\subseteq E^*_{\textnormal{log}}(V)$.

An important property of the complex $E_{\textnormal{log}}^*(V)$ is that it is strictly related to Deligne-Beilinson cohomology. More precisely, let us consider the following
complex
\[
 E^*_{\textnormal{log},\matR}(V,p)_D:= s(u:E^*_{\textnormal{log},\matR}(V,p)\oplus F^p E^*_{\textnormal{log}}(V)\rightarrow E_{\textnormal{log}}^*(V) )
\]
where $u(a,b)=b-a$ and $s(u)$ stands for the simple complex of the map $u$. Then 
\[
 \textnormal{H}^{*}_{D}(V, \matR(p))=\textnormal{H}^*(E^*_{\textnormal{log},\matR}(V,p)_D).
\]
\subsection{Definition of Burgos' arithmetic Chow groups}
Now fix $p\in\matN^*$. Any cycle $y\in Z^p(X)$ defines a class in 
\[ \rho(y)\in
\textnormal{H}^{2p}_{D,Y}(X_{\matR}, \matR(p))
\]
where $Y=\textnormal{supp } y$. Any $g\in E^{p-1,p-1}_{\textnormal{log},\matR}(X(\matC)\setminus Y(\matC),p-1)$                                          
with 
 $\partial \bar{\partial} g\in E^{2p}(X(\matC))$, also
defines a class \[\{-2\partial \bar{\partial} g,g\}\in \textnormal{H}^{2p}_{D,Y}(X, \matR(p)).\]
The space of Green forms associated with $y$ is then
\[ \Scale[0.9]{
 \textnormal{GE}^p_y(X_\matR):=\left.\left\{ g\in   E^{p-1,p-1}_{\textnormal{log},\matR}(X(\matC)
 \setminus Y(\matC),p-1) \left|
 \begin{array}{c}  
 -2\partial \bar{\partial} g\in E^{2p}(X(\matC)) \\
 \{-2\partial \bar{\partial} g,g\}=\rho(y) \\
 F_\infty^*
  g=\bar g 
\end{array}\right.
 \right\}\right/ (\textnormal{Im}\partial +\textnormal{Im} \bar \partial). }
\] 
The group of arithmetic cycles in the sense of Burgos is
\[
 \hat{Z}^p(X):=\left\{ (y,\tilde g)| \ y\in Z^p(X) \textnormal{ and } \tilde g\in\textnormal{GE}^p_y(X_\matR)\right\}.
\]
If $W$ is a codimension $p-1$ irreducible subvariety of $X$ and $f\in k(W)^*$, we have a well-defined
subvariety $W(\matC)$ of $X(\matC)$  and a well-defined function $f_\matC\in k(W(\matC))^*$. To $f_\matC$ is associated a class
\[ 
 \rho(f_\matC)\in \textnormal{H}^{2p-1}_{D}\left((X\setminus F)_\matR, \matR(p)\right)
\]
where $F=\textnormal{supp div} f$. Since $\textnormal{H}^{2p-1}_{D}\left((X\setminus F)_\matR, \matR(p)\right)$
is the same as
\[
 \left.\left\{ g\in  E^{p-1,p-1}_{\textnormal{log},\matR}(X(\matC)\setminus F(\matC),p-1)| \
 \partial \bar{\partial} g=0 \textnormal{ and } 
 F_\infty^*g=\bar g \right\}\right/(\textnormal{Im}\partial +\textnormal{Im} \bar \partial),
\]
then one can check that $\rho(f_\matC)$ defines
an element in $\textnormal{GE}^p_{\textnormal{div} f}(X_\matR)$, which we denote by $ \textnormal{b}(\rho(f_\matC))$. Let
$\widehat{\textnormal{Rat}}^p(X)$ be the subgroup of
$\hat{Z}^p(X)$
generated by the elements of the form
\[
 \widehat{\textnormal{div}}f=\left(\textnormal{div}f,-\textnormal{b}(\rho(f_\matC))\right).
\]
The arithmetic Chow group 
$\widehat{\textnormal{CH}}^p_{\textnormal{log}} (X)$ in the sense of Burgos is 
\[
\widehat{\textnormal{CH}}^p_{\textnormal{log}} (X):=\hat{Z}^p(X)\big/\widehat{\textnormal{Rat}}^p(X).
\]
The class of an element $(y,\tilde g)\in \hat{Z}^p(X)$ will be denoted by $[y,\tilde g]$. 
\subsection{Two important properties}
Burgos' arithmetic Chow groups fit in the following exact sequence
\begin{equation}\label{exbur}
\begin{tikzpicture}[description/.style={fill=white,inner sep=2pt}, ciao/.style={fill=red,inner sep=2pt}]
\matrix (m) [matrix of math nodes, row sep=2.5em,
column sep=2em, text height=1.5ex, text depth=0.25ex]
{  \textnormal{CH}^{p,p-1}(X) & \textnormal{H}^{2p-1}_D(X_{\matR},\matR(p)) &
\widehat{\textnormal{CH}}^p_{\textnormal{log}} (X) & \textnormal{CH}^p (X) 
\oplus Z E^{p,p}_{\textnormal{log}}(X_\matR) \\
  };

	\path[->,font=\scriptsize] 
		(m-1-1) edge node[above] {$ \rho $} (m-1-2)
		(m-1-2) edge node[above] {$ \textnormal{a} $} (m-1-3)
		(m-1-3) edge node[above] {$ (\zeta,-\omega) $} (m-1-4)
		;
\end{tikzpicture}
\end{equation}
where $ \textnormal{CH}^{p,p-1}$ is Gillet-Soulé's version of one of Bloch's higher Chow groups,
$ \rho$ is defined in the proof of Corollary 6.3 in \cite{burgosarith},
the map $a$ sends the class of $\tilde f$ to $[0,\tilde f]$ and $(\zeta,-\omega)([y,\tilde g])=(y,2\partial \bar
\partial g)$. Later we will make use of the fact
\[
 \textnormal{CH}^{p,p-1}(X)_\matQ \cong H^{2p-1}_{\mathcal{M}}(X,p)
\]
(see section 1.4 in \cite{burgoscoho} for this).

Furthermore there exists a homomorphism
\[
 \psi_X:\widehat{\textnormal{CH}}^p_{\textnormal{log}}(X)\rightarrow  \widehat{\textnormal{CH}}^p (X)
\]
which is compatible with pull-backs and is an isomorphism if $X$ is proper over $R$.
In particular for any 
$y\in Z^p(X)$, we have a commutative diagram
\begin{center}
\begin{tikzpicture}[description/.style={fill=white,inner sep=2pt}, ciao/.style={fill=red,inner sep=2pt}]
\matrix (m) [matrix of math nodes, row sep=3.5em,
column sep=3.5em, text height=1.5ex, text depth=0.25ex]
{ \widehat{\textnormal{CH}}^p_{\textnormal{log}} (X)&  \widehat{\textnormal{CH}}^p_{\textnormal{log}} (X\setminus Y) \\
  \widehat{\textnormal{CH}}^p (X) & \widehat{\textnormal{CH}}^p (X\setminus Y) \\
  };
	\path[->,font=\scriptsize]
		(m-1-1) edge node[left] {$ \psi_X $} (m-2-1)
		
		(m-1-2) edge node[auto] {$ \psi_{X\setminus Y} $} (m-2-2)
		
		;
	\path[->,font=\scriptsize] 
		(m-1-1) edge node[above] {$ i^\ast $} (m-1-2)
		(m-2-1) edge node[above] {$ i^\ast $} (m-2-2)
		
		;
\end{tikzpicture}
\end{center}
where $Y=\textnormal{supp }y$ and the map $i$ is the immersion $X\setminus Y \hookrightarrow X$.
\section{\textbf{AN INTERMEDIATE RESULT}}
\begin{theo}\label{genlem}
 Let $X/R$ be a proper arithmetic variety, $y$ any cycle in $Z^p(X)$ and $h$ any Green current for $y$. Then there exists a representative $h_0$ of $h$ belonging to
 $E^{p-1,p-1}_{\textnormal{log},\matR}(X(\matC)\setminus Y(\matC))$, where $Y=\textnormal{supp}\ y$. 
 If $\omega([y,h]):=\delta_y+\textnormal{dd}^{\textnormal{c}} h=0$, then $\frac{h_0}{2(2\pi i)^{1-p}}$ defines a class 
 \[\frac{\widetilde{h_0}}{2(2\pi i)^{1-p}}\in  
\textnormal{H}^{2p-1}_{D}((X\setminus Y)_\matR, \matR(p))\] which does not depend on the choice of $h_0$ and verifies
\[
 a\left(\frac{\widetilde{h_0}}{2(2\pi i)^{1-p}}\right)=i^*\left(\psi_{X}^{-1}([y,h])\right).
\]
\end{theo}
\begin{proof}
If we denote by $\text{GC}_X(y)$ the space of Green currents for the cycle $y$, Theorem 5.9 in \cite{burgosarith}
tells us that there is a natural isomorphism
\[
 \textnormal{GE}^p_y(X_\matR)\rightarrow GC_X(y)
\]
which sends $\tilde g $ to the class of the current $2 (2\pi i )^{1-p} [g]$, where $[g]$ sends a form $\omega$ to
\[
 [g](\omega)=\int_{X(\matC)} \omega\wedge g.
\]
Let $\widetilde{h_\textnormal{log}}\in \textnormal{GE}^p_y(X_\matR)$ be the inverse image of $h$ by this isomorphism. Any element in $2 (2\pi i )^{1-p} \widetilde{h_\textnormal{log}}$
gives a representative of $h$ belonging to $E^{p-1,p-1}_{\textnormal{log},\matR}(X(\matC)\setminus Y(\matC))$.

If $\delta_y+\textnormal{dd}^{\textnormal{c}} h=0$, we deduce
\[
0= (\delta_y+\textnormal{dd}^{\textnormal{c}} h)_{|_{X(\matC)\setminus Y(\matC)}}=
(\textnormal{dd}^{\textnormal{c}} h)_{|_{X(\matC)\setminus Y(\matC)}}=2 (2\pi i )^{1-p} \textnormal{dd}^{\textnormal{c}}
\left(\widetilde{h_{\textnormal{log}}}\right).
\]
Therefore we have a well-defined element $\widetilde{h_\textnormal{log}}$ in 
\[
 \left.\left\{ g\in   E^{p-1,p-1}_{\textnormal{log},\matR}(X(\matC)
 \setminus Y(\matC),p-1) \left|
 \begin{array}{c}  
 \partial \bar{\partial} g=0 \\
 F_\infty^*
  g=\bar g 
\end{array}\right.
 \right\}\right/ (\textnormal{Im}\partial +\textnormal{Im} \bar \partial)  
\]
i.e. in $\textnormal{H}^{2p-1}_{D}((X\setminus Y)_\matR, \matR(p))$. \\
By definition of $\widetilde{h_{\textnormal{log}}}$ and by definition of the pullback in Burgos' theory, we have
\[
i^* \left(\psi_{X}^{-1}([y,h])\right)=i^* \left(\left[y, \widetilde{h_{\textnormal{log}}}\right]\right)=
\left[0,\widetilde{h_{\textnormal{log}}}\right]=
a \left(\widetilde{h_{\textnormal{log}}}\right).
\]
To conclude the proof, take any representative $h_0$ of $h$ in
$E^{p-1,p-1}_{\textnormal{log},\matR}(X(\matC)\setminus Y(\matC))$. We want to show that $\frac{h_0}{2(2\pi i)^{1-p}}$ defines a class 
$\frac{\widetilde{h_0}}{2(2\pi i)^{1-p}}$ in 
$\textnormal{H}^{2p-1}_{D}((X\setminus Y)_\matR, \matR(p))$ and that this class is $\widetilde{h_\textnormal{log}}$. Since 
$\delta_y+\textnormal{dd}^{\textnormal{c}} h_0=0$, Proposition 6.5 in \cite{burgoscoho} implies that
$\frac{\widetilde{h_0}}{2(2\pi i)^{1-p}}$ is an element in  
$\textnormal{GE}^p_y(X_\matR)$ (please notice that the definitions of current associated to a cycle and of current associated to a differential form used in Burgos' paper
slightly differ from the ones used here). This element
must coincide with $\widetilde{h_\textnormal{log}}$, since both have image $h$ in  $\text{GC}_X(y)$.
\end{proof}
\begin{rem}\label{image}
 Tensoring with $\matQ$ over $\matZ$ is an exact operation, so the exact sequence (\ref{exbur}) shows that
 \begin{align*}
 i^* \left(\psi_{X}^{-1}([y,h])\right)=0\in \widehat{\textnormal{CH}}^p_{\textnormal{log}} (X\setminus Y)_\matQ
 \  & \Leftrightarrow \ \frac{\widetilde{h_0}}{2(2\pi i)^{1-p}}\in \rho
 \left( \textnormal{CH}^{p,p-1}(X\setminus Y)_{\matQ} \right)\\
 & \Leftrightarrow \ \frac{\widetilde{h_0}}{2(2\pi i)^{1-p}}\in \textnormal{cyc}
 \left( H^{2p-1}_{\mathcal{M}}(X\setminus Y,p) \right)
\end{align*}
\end{rem}
\section{\textbf{THE CASE OF AN ABELIAN SCHEME OVER AN ARITHMETIC VARIETY}}
In this section we explain in detail how the class
\[
 ([N]^\ast \mathfrak{g}_{\mathcal{A}^\vee}
 -N^{2g} \mathfrak{g}_{\mathcal{A}^\vee})|_{\mathcal{A}\setminus\mathcal{A}[N]}
\]
actually gives rise to an element in $ \textnormal{H}^{2g-1}_{D,\textnormal{an}}((\mathcal{A}\setminus\mathcal{A}[N])_{\matR},\matR(g))$.

Let $S$ be an arithmetic variety over $R$ and let $\pi:\mathcal{A}\rightarrow S$ be an abelian scheme 
over $S$ of relative dimension $g$. We shall write as usual 
$\mathcal{A}^\vee\rightarrow S$ for the dual abelian scheme. We let $\epsilon_\mathcal{A}=\epsilon$ be the 
zero-section of $\pi:\mathcal{A}\rightarrow S$. 
We shall denote by $S_0=S_{0,\mathcal{A}}=\epsilon_\mathcal{A}(S)$ the reduced 
closed subscheme of $\mathcal{A}$, which is the image of $\epsilon_\mathcal{A}$. 
We write $\mathcal{P}$ for the Poincaré bundle 
on $\mathcal{A}\times_S \mathcal{A}^\vee$ and $p_1:\mathcal{A}\times_S \mathcal{A}^\vee \rightarrow \mathcal{A}$ 
for the first projection.
Since $\mathcal{A}$ and $S$ are arithmetic varieties over $R$, 
we have two well-defined complex manifolds
$\mathcal{A}(\matC)$ and $S(\matC)$ and two well-defined $\matR$-vector spaces $D^{g-1,g-1}(\mathcal{A}_\matR)$ and $E^{g-1,g-1}(\mathcal{A}_\matR)$.
We endow the Poincaré bundle $\mathcal{P}$ with the unique metric 
$h_\mathcal{P}$ such that the canonical rigidification of $\mathcal{P}$ along the zero-section 
$\epsilon\times \textnormal{Id}_{\mathcal{A}^\vee}:\mathcal{A}^\vee\rightarrow 
\mathcal{A}\times_S \mathcal{A}^\vee$ is an isometry and such that the curvature form of $h_\mathcal{P}$ is translation 
invariant along the fibres of the map $\mathcal{A}(\mathbb{C})\times_{S(\mathbb{C})} \mathcal{A}^\vee (\matC)
\rightarrow\mathcal{A}^\vee (\matC)$. We write 
$\bar{\mathcal{P}}=(\mathcal{P},h_\mathcal{P})$ for the resulting hermitian line bundle.

\begin{theo}\label{uniqclass}
 There is a unique class of currents $\mathfrak{g}_{\mathcal{A}^\vee} \in \tilde{D}^{g-1,g-1}(\mathcal{A}_\matR)$ with the 
 following three properties:
 \begin{enumerate}
  \item Any element of $\mathfrak{g}_{\mathcal{A}^\vee}$ is a Green current for $S_0(\matC)$.
  \item The identity  $(S_0,\mathfrak{g}_{\mathcal{A}^\vee})=(-1)^g p_{1,\ast}
  (\widehat{\textnormal{ch}}(\bar{\mathcal{P}}))^{(g)} $ holds in $\widehat{\textnormal{CH}}^g (\mathcal{A})_\matQ $.
  \item The identity $\mathfrak{g}_{\mathcal{A}^\vee}=[n]_\ast \mathfrak{g}_{\mathcal{A}^\vee}$ holds  for all $n\geq 2$.
 \end{enumerate}
\end{theo}
Here $[n]:\mathcal{A}\rightarrow \mathcal{A}$ is the multiplication-by-$n$ morphism and $\widehat{\textnormal{ch}}(\bar{\mathcal{P}})$ is the arithmetic Chern character of the hermitian
bundle $\bar{\mathcal{P}}$.

Let us fix $N$ a positive natural number and call
\[ T:=[N]^*(S_0,  \mathfrak{g}_{\mathcal{A}^\vee})-N^{2g}(S_0,  \mathfrak{g}_{\mathcal{A}^\vee}). \] 
We will write $[T]$ for its class in $\widehat{\textnormal{CH}}^g (\mathcal{A})$. We denote by $\mathcal{A}[N]$ 
the $N$-torsion of $\mathcal{A}$, by $i$ the immersion  $\mathcal{U}:=\mathcal{A}\setminus\mathcal{A}[N]\hookrightarrow \mathcal{A}$ 
and by $\Gamma$ 
the restriction of the class of  currents 
 \[[N]^\ast \mathfrak{g}_{\mathcal{A}^\vee}
 -N^{2g} \mathfrak{g}_{\mathcal{A}^\vee} \] 
 to $\mathcal{U}(\matC)$.
Then the morphism
\[
  i^*:\widehat{\textnormal{CH}}^g (\mathcal{A}) \rightarrow \widehat{\textnormal{CH}}^g (\mathcal{U})
\]
sends $[T]$ to $[0,\Gamma]$.
We recall the fundamental exact sequence
\begin{equation}\label{exact}
\begin{tikzpicture}[description/.style={fill=white,inner sep=2pt}, ciao/.style={fill=red,inner sep=2pt}]
\matrix (m) [matrix of math nodes, row sep=3.5em,
column sep=2.5em, text height=1.5ex, text depth=0.25ex]
{  \textnormal{CH}^{g,g-1}(\mathcal{U}) & \tilde E^{g-1,g-1}(\mathcal{U}_\matR) &  \widehat{\textnormal{CH}}^g (\mathcal{U})
& \textnormal{CH}^g (\mathcal{U}) & 0\\
  };
	\path[->,font=\scriptsize] 
		(m-1-1) edge node[above] {$ \rho_{\textnormal{an}}$ } (m-1-2)
		(m-1-2) edge node[above] {$ a $} (m-1-3)
		(m-1-3) edge node[above] {$  $} (m-1-4)
		(m-1-4) edge node[above] {$ $} (m-1-5)
		;
\end{tikzpicture}
\end{equation}
(See Theorem 3.3.5 in \cite{gilsou}) for this). Here the map $a$ sends the class of $\omega$ to $[0,\omega]$. By construction, $ \rho_{\textnormal{an}}$ is the 
following composite function
\begin{center}
\begin{tikzpicture}[description/.style={fill=white,inner sep=2pt}, ciao/.style={fill=red,inner sep=2pt}]
\matrix (m) [matrix of math nodes, row sep=2.5em,
column sep=2.5em, text height=1.5ex, text depth=0.25ex]
{  \textnormal{CH}^{g,g-1}(\mathcal{U}) & \textnormal{H}^{2g-1}_{D}(\mathcal{U}_{\matR},\matR(g)) &
\textnormal{H}^{2g-1}_{D,\textnormal{an}}(\mathcal{U}_{\matR},\matR(g)) &\tilde E^{g-1,g-1}(\mathcal{U}_\matR) \\
  };

	\path[->,font=\scriptsize] 
		(m-1-1) edge node[above] {$ \rho $} (m-1-2)
		(m-1-2) edge node[above] {$ \textnormal{forgetful} $} (m-1-3)
		(m-1-3) edge node[above] {$  $} (m-1-4)
		;
\end{tikzpicture}
\end{center}
where the third map is a natural inclusion. Indeed we have
\[
 \textnormal{H}^{2g-1}_{D,\textnormal{an}}(\mathcal{U}_{\matR},\matR(g))=\left\{ c\in   (2\pi i)^{g-1}E^{g-1,g-1}(\mathcal{U}_{\matR}) \left| 
 \partial \bar{\partial} g=0 
 \right\}\right/ (\textnormal{Im}\partial +\textnormal{Im} \bar \partial)  
\]
and the class of $c$ is sent to the class of $c/(2\pi i)^{g-1} $.

Since the image of $[0,\Gamma]$ in $\textnormal{CH}^g (\mathcal{U})$ is $0$, there exists an element in
$\tilde E^{g-1,g-1}(\mathcal{U}_\matR)$ sent to $[0,\Gamma]$ by $a$. We know how to construct such an element:
thanks to Theorem 
 1.3.5 and Theorem 1.2.4 in \cite{gilsou}, the class of currents 
 $[N]^*\mathfrak{g}_{\mathcal{A}^\vee}-N^{2g}\mathfrak{g}_{\mathcal{A}^\vee}$
 has a representative $\gamma$ in $E^{g-1,g-1}(\mathcal{U}_\matR)$.
 Now if $\gamma'$ is any another representative smooth on $\mathcal{U}(\matC)$, we have that
 $\gamma-\gamma'$ is a smooth form on $\mathcal{U}(\matC)$ and $\gamma-\gamma'=
 \partial c_1+ \bar\partial c_2$ for some currents $c_1$ and $c_2$. Therefore there exist two smooth forms 
$\omega_1$ and $\omega_2$ such that $\gamma-\gamma'=
\partial \omega_1 + \bar \partial \omega_2$ (see Theorem 1.2.2 (ii) in \cite{gilsou}). This implies that the class
$\tilde \gamma\in \tilde E^{g-1,g-1}(\mathcal{U}_\matR)$ does not depend on $\gamma$.
We have
\[
 a(\tilde \gamma)=[0,\tilde \gamma]=[0,\Gamma].
\]
The calculations we do to prove the next proposition are basically the same the reader can find in  Lemma 2.4.5 in \cite{kingsross}.
\begin{prop} \label{zero}
 The element $[T]$ is zero in $\widehat{\textnormal{CH}}^g (\mathcal{A})_\matQ$.
\end{prop}
\begin{proof}
 By Theorem \ref{uniqclass}, we have
 \[
  [T]=(-1)^g ([N]^*-N^{2g})\left( p_{1,*}\left(\widehat{\textnormal{ch}}(\bar{\mathcal{P}})\right)^{(g)}\right)
 \]
in $\widehat{\textnormal{CH}}^g (\mathcal{A})_\matQ$. Now consider the following square of schemes over $S$:
\begin{center}
\begin{tikzpicture}[description/.style={fill=white,inner sep=2pt}, ciao/.style={fill=red,inner sep=2pt}]
\matrix (m) [matrix of math nodes, row sep=3.5em,
column sep=3.5em, text height=1.5ex, text depth=0.25ex]
{ \mathcal{A}\times_S \mathcal{A}^\vee & \mathcal{A}\times_S \mathcal{A}^\vee  \\
  \mathcal{A} & \mathcal{A} \\
  };
	\path[->,font=\scriptsize]
		(m-1-1) edge node[left] {$ p_1 $} (m-2-1)
		
		(m-1-2) edge node[auto] {$ p_1 $} (m-2-2)
		
		;
	\path[->,font=\scriptsize] 
		(m-1-1) edge node[above] {$ N\times \textnormal{Id} $} (m-1-2)
		(m-2-1) edge node[above] {$ N $} (m-2-2)
		
		;
\end{tikzpicture}
\end{center}
For any $Q$ scheme over $S$ and any pair of morphisms of schemes over $S$ 
\begin{align*} (\zeta,\sigma)&:Q\rightarrow 
\mathcal{A}\times_S \mathcal{A}^\vee \\ \eta&:Q\rightarrow \mathcal{A}
\end{align*} 
 we can define 
$(\eta,\sigma):Q\rightarrow \mathcal{A}\times_S \mathcal{A}^\vee$. It is easy to see that this is a morphism of schemes 
over $S$ and that, if $N\circ \eta=\zeta$, it verifies
$(N\times \textnormal{Id})\circ (\eta,\sigma)=(\zeta,\sigma)$ and $p_1\circ(\eta,\sigma)=\eta$. Furthermore, $(\eta,\sigma)$ is the 
unique morphism of schemes over $S$ with these properties. Therefore the square above is cartesian. Since the direct image in arithmetic Chow theory is naturally
compatible with smooth base change, we have
\begin{align*}
 (-1)^g[N]^\ast \left( p_{1,\ast}
  (\widehat{\textnormal{ch}}(\bar{\mathcal{P}}))^{(g)}\right)
  &=(-1)^g p_{1,\ast}
  \left[\left((N\times \textnormal{Id})^\ast \widehat{\textnormal{ch}}(\bar{\mathcal{P}})\right)^{(2g)}\right] \\
  &=(-1)^g p_{1,\ast}\left[
  \left(\widehat{\textnormal{ch}}((N\times \textnormal{Id})^\ast (\bar{\mathcal{P}}))\right)^{(2g)}\right]. \\
\end{align*}
From the definition of dual abelian scheme we know that there is an isomorphism between the group
$\textnormal{End}(\mathcal{A}, \mathcal{A})$ and the group of isomorphism classes of invertible sheaves 
$\mathcal{L}$ on
$\mathcal{A}\times_S \mathcal{A}^\vee$  with rigidification along $\epsilon_{\mathcal{A}^\vee}$ such that 
$\mathcal{L}\otimes
k(a)$ is algebraically equivalent to $0$ in $A_a$ for all $a\in \mathcal{A}^\vee$. Via this isomorphism, a map $\phi:\mathcal{A}\rightarrow \mathcal{A}$  
is sent to $(\phi \times id_{\mathcal{A}^\vee})^*(\mathcal{P})$, so the image 
of $\textnormal{Id}_{\mathcal{A}}$ is $\mathcal{P}$. Then the image of $[N]$, i.e $(N\times \textnormal{Id})^\ast (\mathcal{P})$, must coincide with  
$\mathcal{P}^{\otimes N}$. Therefore:
\[ (-1)^g[N]^\ast \left( p_{1,\ast}
  (\widehat{\textnormal{ch}}(\bar{\mathcal{P}}))^{(g)}\right)=(-1)^g p_{1,\ast}\left[
  \left(\widehat{\textnormal{ch}}\left(\bar{\mathcal{P}}^{\otimes N}\right)\right)^{(2g)}\right]
  \]
and
\begin{align*}
 [T]&=(-1)^g \left[   p_{1,\ast}\left[
  \left(\widehat{\textnormal{ch}}\left(\bar{\mathcal{P}}^{\otimes N}\right)\right)^{(2g)}\right] -N^{2g} 
  p_{1,*}\left(\widehat{\textnormal{ch}}(\bar{\mathcal{P}})\right)^{(g)}\right] \\
  &=(-1)^g \left[p_{1,\ast}
  \left(\frac{\hat{\textnormal{c}_1} \left(\bar{\mathcal{P}}^{\otimes N}\right)^{2g}}{2g!}\right)-N^{2g} p_{1,\ast}
  \left(\frac{\hat{\textnormal{c}_1}
  (\bar{\mathcal{P}})^{2g}}{2g!}\right)\right]\\
  &=(-1)^g \left[p_{1,\ast}
  \left(\frac{N^{2g}\hat{\textnormal{c}_1} \left(\bar{\mathcal{P}}\right)^{2g}}{2g!}\right)-N^{2g} p_{1,\ast}
  \left(\frac{\hat{\textnormal{c}_1}
  (\bar{\mathcal{P}})^{2g}}{2g!}\right)\right] \\
  &=0 \\
\end{align*}
where $\hat{\textnormal{c}_1}(\cdot) $ refers to the first  arithmetic Chern class of a hermitian bundle. Notice that we  used the multiplicativity of 
$ \hat{\textnormal{c}_1}(\cdot)$.
\end{proof}
\begin{cor}
The class $(2\pi i)^{g-1} \tilde \gamma \in (2\pi i)^{g-1}\tilde E^{g-1,g-1}(\mathcal{U}_\matR)$ is in the image of 
\[ \textnormal{cyc}_{\textnormal{an}}:  H^{2g-1}_{\mathcal{M}}(\mathcal{U},g)\rightarrow \textnormal{H}^{2g-1}_{D,\textnormal{an}}(\mathcal{U},\matR(g)) . \]
\end{cor}
\begin{proof}
 Proposition \ref{zero} implies that \[ 0=i^*([T])=[0,\Gamma]=a(\tilde \gamma)\] in  $\widehat{\textnormal{CH}}^g (\mathcal{U})_\matQ $ and the 
 exactness of sequence (\ref{exact}) gives us 
 $ \tilde \gamma \in \rho_{\textnormal{an}}( \textnormal{CH}^{g,g-1}(\mathcal{U})_{\matQ}). $
By the definition of $ \rho_{\textnormal{an}}$, we obtain that 
\[
 (2\pi i)^{g-1} \tilde \gamma\in (\textnormal{forgetful}\circ \rho)(\textnormal{CH}^{g,g-1}(\mathcal{U})_{\matQ}).
\]
This is exactly what we wanted to prove, once one identifies $ H^{2g-1}_{\mathcal{M}}(\mathcal{U},g) $ with $ \textnormal{CH}^{g,g-1}(\mathcal{U})_{\matQ} $.
\end{proof}
\section{\textbf{THE MAIN RESULT}}
We are now able to apply Theorem \ref{genlem} to our situation. We assume that $S$ is proper over
$R$, so $\mathcal{A} $ is proper over $R$.
\begin{cor}\label{poly}
 There exists a representative of $[N]^*\mathfrak{g}_{\mathcal{A}^\vee}-N^{2g}\mathfrak{g}_{\mathcal{A}^\vee}$
 belonging to $E^{g-1,g-1}_{\textnormal{log},\matR}(\mathcal{U}(\matC))$. Any such $\eta$
   defines an element
 \[\frac{\widetilde{\eta}}{2(2\pi i)^{1-g}}\in \textnormal{Im}\left(\textnormal{cyc}: 
 H^{2g-1}_{\mathcal{M}}(\mathcal{U},g) \rightarrow \textnormal{H}^{2g-1}_{D}
 (\mathcal{U}_\matR, \matR(g))\right)
 \]
which does not depend on the choice of $\eta$ and verifies
\[
 a\left(\frac{\widetilde{\eta}}{2(2\pi i)^{1-g}}\right)=i^*\left(\psi_{\mathcal{A}}^{-1}([y,h])\right)=0.
\]
Furthermore
\[
 \textnormal{forgetful}\left(\frac{\widetilde{\eta}}{2(2\pi i)^{1-g}}\right)= \frac{\widetilde{\gamma}}{2(2\pi i)^{1-g}}.
\]
\end{cor}
\begin{proof}
 To prove the first assertion we apply Theorem \ref{genlem} and Remark \ref{image} with $X$ equal to $\mathcal{A}$ and 
 $(y,h)$ equal to $T$. The hypotheses are satisfied since  $[T]=0$ in $\widehat{\textnormal{CH}}^g (\mathcal{A})_\matQ$, so 
 $i^*\left(\psi_{\mathcal{A}}^{-1}([T])\right)=0$ in $\widehat{\textnormal{CH}}^g_{\textnormal{log}} 
 (\mathcal{U})_\matQ$. 
 
 To prove the second assertion, it is enough to notice that any 
 representative $\eta$ of $[N]^*\mathfrak{g}_{\mathcal{A}^\vee}-N^{2g}\mathfrak{g}_{\mathcal{A}^\vee}$ 
 belonging to $E^{g-1,g-1}_{\textnormal{log},\matR}(\mathcal{U}(\matC))$, also belongs to  $E^{g-1,g-1}(\mathcal{U}_\matR)$.
\end{proof}
\begin{rem}\label{weight0}
 A natural analog of the operator $\textnormal{tr}_{[a]}$ operates on Deligne-Beilinson cohomology and the map $\textnormal{cyc}$ intertwines this operator with 
 $\textnormal{tr}_{[a]}$. Therefore from the existence of the Jordan  decomposition and the fact
 \[\frac{\widetilde{\eta}}{2(2\pi i)^{1-g}}\in \textnormal{Im}\left(\textnormal{cyc}: 
 H^{2g-1}_{\mathcal{M}}(\mathcal{U},g) \rightarrow \textnormal{H}^{2g-1}_{D}
 (\mathcal{U}_\matR, \matR(g))\right)
 \]
 we deduce 
  \[\frac{\widetilde{\eta}}{2(2\pi i)^{1-g}}\in \textnormal{Im}\left(\textnormal{cyc}: 
 H^{2g-1}_{\mathcal{M}}(\mathcal{U},g)^{(0)} \rightarrow \textnormal{H}^{2g-1}_{D}
 (\mathcal{U}_\matR, \matR(g))\right).
 \]
\end{rem}
We are ready to prove our main result.
\begin{theo} \label{realDB} Let $\textnormal{pol}^0\in H^{2g-1}_{\mathcal{M}}(\mathcal{U},g)$ be the zero step of the motivic polylogarithm on
$\mathcal{A}$ (as defined in the Introduction). Then
\[
 -2\cdot\textnormal{cyc}(\textnormal{pol}^0)= \frac{\tilde\eta}{(2\pi i)^{1-g}}
\]
in $\textnormal{H}^{2g-1}_{D}
 (\mathcal{U}_\matR, \matR(g))$.
\end{theo}
\begin{proof} 
 We start noticing that 
 \[
 \text{cyc}_{\textnormal{an}}(\textnormal{pol}^0)=- \frac{\tilde\gamma}{2(2\pi i)^{1-g}}= \textnormal{forgetful}\left(- \frac{\tilde\eta}{2(2\pi i)^{1-g}}\right).
 \]
By Remark \ref{weight0} we know that 
 \[
  - \frac{\tilde\eta}{2(2\pi i)^{1-g}}= \textnormal{cyc}(l)
 \]
for some $l \in H^{2g-1}_{\mathcal{M}}(\mathcal{U},g)^{(0)}$. Therefore we have 
\[
 \text{cyc}_{\textnormal{an}}(\textnormal{pol}^0)= \text{cyc}_{\textnormal{an}}(l)
\]
and the injectivity of $\text{cyc}_{\textnormal{an}} $ on $H^{2g-1}_{\mathcal{M}}(\mathcal{U},g)^{(0)}$ implies that $\textnormal{pol}^0=l$. We then obtain
 \[
  \textnormal{cyc}(\textnormal{pol}^0)=\textnormal{cyc}(l)= - \frac{\tilde\eta}{2(2\pi i)^{1-g}}.
 \]

\end{proof}
\nocite{*}
\bibliography{biblioarak}{}
\bibliographystyle{alpha}

\end{document}